\documentclass[12pt]{article}
\usepackage{pifont}

\usepackage{amsfonts}
\usepackage{latexsym}
\usepackage{amsmath}
\usepackage{amssymb}
\usepackage{color}

\newtheorem{theorem}{Theorem}[section]

\newenvironment{proof}[1][Proof]{\noindent \textbf{#1.} }{\ \ \  $\Box$}

\newtheorem{lemma}{Lemma}[section]
\newtheorem{definition}{Definition}[section]
\newtheorem{remark}{Remark}[section]

\title{Fully Coupled Forward-Backward Stochastic Functional Differential Equations and Applications to Quadratic Optimal Control \thanks{This work is supported by the National Natural Science Foundation of China (Grant No. 11301274), the Mathematical Tianyuan Foundation of
China (Grant No. 11126050), the Specialized Research Fund for the
Doctoral Program of Higher Education of China (Grant No.
20113207120002), and Program of Natural Science Research of Jiangsu
Higher Education Institutions of China (Grant No. 13KJB110017).}}

\date{}

 \author{ Xiaoming Xu\thanks{E-mail: xmxu@njnu.edu.cn}
 \\ \small{Institute of Finance and Statistics, School of Mathematical Sciences,} \\
 \small{Nanjing Normal University, Nanjing, 210023, China}
 }

\begin{document}

\maketitle

\begin{abstract}
In this paper, we consider the fully coupled forward-backward
stochastic functional differential equations (FBSFDEs) with
stochastic functional differential equations as the forward
equations and the generalized anticipated backward stochastic
differential equations as the backward equations. We will prove the
existence and uniqueness theorem for FBSFDEs. As an application, we
deal with a quadratic optimal control problem for functional
stochastic systems, and get the explicit form of the optimal control
by virtue of FBSFDEs.
\\
\par $\textit{Keywords:}$ stochastic functional differential equation, generalized anticipated backward stochastic
differential equation, forward-backward stochastic functional
differential equation, quadratic optimal control, functional
stochastic system
\end{abstract}



\section{Introduction}

Backward stochastic differential equation (BSDE) was considered the
general form the first time by Pardoux-Peng \cite{PP1} in 1990. In
the last twenty years, the theory of BSDEs has been studied with
great interest (see e.g. \cite{KPQ, PW, PY}). One hot topic is the
forward-backward stochastic differential equation (FBSDE) (see e.g.
\cite{HP, MPY, Y}), due to its wide applications in the
pricing/hedging problem, in the stochastic control and game theory
(see e.g. \cite{CW, CM, PW, W0, W, YJ}).

In the previous results, the FBSDE is mostly of the form
\begin{equation*}
\left\{
\begin{tabular}{l}
$dX_t=b(t, X_t, Y_t, Z_t)dt+\sigma(t, X_t, Y_t, Z_t)dB_t;$ \vspace{2mm}\\
$-dY_t=f(t, X_t, Y_t, Z_t)dt-Z_tdB_t;$ \vspace{2mm}\\
$X_0=a$ \qquad $Y_T=\Phi(X_T),$
\end{tabular}\right.
\end{equation*}
where the forward SDE is the state equation, and the BSDE is the
dual equation appearing in the control system.

However, many natural and social phenomena shows that the state
process at time $t$ depends not only on its present state but also
its past history. Similarly, for the dual process, its value at time
$t$ depends not only on its present value but also its future value.
Also motivated by the work of Peng and Yang \cite{PY}, recently Chen
and Wu \cite{CW} studied the following general FBSDE:
\begin{equation*}
\left\{
\begin{tabular}{ll}
$dX_t=b(t, X_t, Y_t, Z_t, X_{t-\delta})dt+\sigma(t, X_t, Y_t, Z_t, X_{t-\delta})dB_t,$ & $t\in [0,T];$ \vspace{2mm}\\
$-dY_t=f(t, X_t, Y_t, Z_t, Y_{t+\delta}, Z_{t+\delta})dt-Z_tdB_t,$ & $t\in [0, T];$ \vspace{2mm}\\
$X_t= \rho_t,$ & $t\in[-\delta, 0];$ \vspace{2mm} \\
$Y_T=\Phi(X_T),$ \quad $Y_t =\xi_t$ & $t\in(T, T+\delta];$
\vspace{2mm} \\
$Z_t= \eta_t,$ & $t\in[T, T+\delta],$
\end{tabular}\right.
\end{equation*}
where $\delta \geq 0$, and the BSDE, as the dual equation, is just
of the form considered in \cite{PY}.

Easily we can find that the case that Chen and Wu \cite{CW}
considered is only a special case, where the value at time $t$
depends on that at time point $t$ and at another time point
$t-\delta$ (or $t+\delta$), that is to say, the influence brought by
the other time intervals is ignored.

Hence, it is necessary for us to study the following general case:
\begin{equation*}
\left\{
\begin{tabular}{rlll}
$dX_t$ &=& $b(t, \{X_r\}_{r\in [-M, t]}, Y_t, Z_t)dt + \sigma(t, \{X_r\}_{r\in [-M, t]}, Y_t, Z_t)dB_t,$ & $t\in [0, T];$ \vspace{2mm}\\
$-dY_t$ &=& $f(t, X_t, \{Y_r\}_{r\in [t, T+K]}, \{Z_r\}_{r\in [t,
T+K]})dt - Z_t dB_t, $ & $ t\in[0, T];$ \vspace{2mm}
\\
$X_t$ &=& $\rho_t,$ & $t\in [-M, 0];$ \vspace{2mm}
\\
$Y_t$ &=& $\xi_t,$\ \ \ $Z_t\ \ =\ \ \eta_t,$ & $t\in[T, T+K]$
\end{tabular}\right.
\end{equation*}
with $M \geq 0$ and $K \geq 0$, where the state process and the dual
process are given in the form of stochastic functional differential
equations (see e.g. Mohammed \cite{M}) and generalized anticipated
BSDEs respectively, and the latter is just the new type of BSDEs
studied by Yang \cite{Y} (see also Yang and Elliott \cite{YE}).

We prove that under proper assumptions, the solution of the above
equation exists uniquely (see Section $3$). Then in Section $4$, as
an application, we deal with an optimal control problem for the
following functional stochastic system:
\begin{equation*}
\left\{
\begin{tabular}{rlll}
$dX_t$ &=& $(A_t \int_{-M}^t X_sds + C_t v_t)dt + (D_t \int_{-M}^t X_sds + F_t v_t) dB_t,$ & $t\in [0, T];$ \vspace{2mm}\\
$X_t$ &=& $\rho_t,$ & $t\in [-M, 0],$
\end{tabular}\right.
\end{equation*}
where $v_\cdot$ is a control process. Our aim is to minimize the
classical quadratic optimal control cost function. For this problem,
we can get the explicit unique optimal control by virtue of the
results obtained in the previous section.

Next we first make some preliminaries.

\section{Preliminaries}

Let $\{B_t; t\geq 0\}$ be a $d$-dimensional standard Brownian motion
on a probability space $(\Omega, \mathcal{F}, P)$ and
$\{\mathcal{F}_t; t\geq 0\}$ be its natural filtration. Denote by
$|\cdot|$ the norm in $\mathbb{R}^n$, and $\langle \cdot, \cdot
\rangle$ denotes the inner product. Given $T
>0,$ we will use the following notations:
\begin{itemize}
\item{$C(-M, 0; \mathbb{R}^n)$ := $\{\varphi_\cdot: [-M, 0] \rightarrow \mathbb{R}^n$ $|$ $\varphi_\cdot$ satisfies $\sup_{-M \leq t \leq 0} |\varphi_t|< +\infty\};$}

\item{$L^2(\mathcal{F}_T; \mathbb{R}^n)$ := $\{\xi\in \mathbb{R}^n$
$|$ $\xi$ is an $\mathcal{F}_T$-measurable random variable such that
$E|\xi|^2< + \infty\};$}

\item{$L_{\mathcal{F}}^2(0, T; \mathbb{R}^n)$ := $\{ \varphi_\cdot:
\Omega\times [0, T]\rightarrow \mathbb{R}^n$ $|$ $\varphi_\cdot$ is
an $\mathcal{F}_t$-progressively measurable process such that
$E\int_0^T |\varphi_t|^2dt< + \infty\}.$}
\end{itemize}

\subsection{Generalized Anticipated Backward Stochastic Differential Equations}

Consider the following generalized anticipated backward stochastic
differential equation (GABSDE):
\begin{equation}\label{chap5.equation:Yang}
\left\{
\begin{tabular}{rlll}
$-dY_t$ &=& $f(t, \{Y_r\}_{r\in [t, T+K]}, \{Z_r\}_{r\in [t,
T+K]})dt-Z_tdB_t, $ & $
t\in[0, T];$\\[1.5mm]
$Y_t$ &=& $\xi_t, $ & $t\in[T, T+K];$\\[1.5mm]
$Z_t$ &=& $\eta_t, $ & $t\in[T, T+K].$
\end{tabular}\right.
\end{equation}

For the generator $f(\omega, t, \{y_r\}_{r\in [t, T+K]},
\{z_r\}_{r\in [t, T+K]}): \Omega \times [0, T]\times
L_\mathcal{F}^2(t, T+K; \mathbb{R}^m)\times L_\mathcal{F}^2(t, T+K;
\mathbb{R}^{m\times d})\rightarrow L^2 (\mathcal{F}_t;
\mathbb{R}^m),$ we use several hypotheses (see Yang \cite{Y}):

${\bf{(A2.1)}}$ There exists a constant $L > 0$ such that for each
$t\in [0, T],$ $y_\cdot, y_\cdot^\prime \in L_\mathcal{F}^2(0, T+K;
\mathbb{R}^m),$ $z_\cdot, z_\cdot^\prime \in L_\mathcal{F}^2(0, T+K;
\mathbb{R}^{m\times d}),$ the following holds:
\begin{align*}
& E[\int_t^T |f(s, \{y_r\}_{r\in [s, T+K]}, \{z_r\}_{r\in [s,
T+K]})-f(s, \{y_r^\prime\}_{r\in [s, T+K]}, \{z_r^\prime\}_{r\in [s,
T+K]})|^2ds] \\
& \leq L E[\int_t^{T+K}(|y_s-y_s^\prime|^2+|z_s-z_s^\prime|^2])ds];
\end{align*}

${\bf{(A2.1^\prime)}}$ There exists a constant $L^\prime > 0$ such
that for each $t\in [0, T],$ $y_\cdot, y_\cdot^\prime \in
L_\mathcal{F}^2(0, T+K; \mathbb{R}^m),$ $z_\cdot, z_\cdot^\prime \in
L_\mathcal{F}^2(0, T+K; \mathbb{R}^{m\times d}),$ the following
holds:
\begin{align*}
& E[\int_t^T e^{\theta s}|f(s, \{y_r\}_{r\in [s, T+K]},
\{z_r\}_{r\in [s, T+K]})-f(s, \{y_r^\prime\}_{r\in [s, T+K]},
\{z_r^\prime\}_{r\in [s,
T+K]})|^2ds] \\
& \leq L^\prime E[\int_t^{T+K}e^{\theta
s}(|y_s-y_s^\prime|^2+|z_s-z_s^\prime|^2])ds],
\end{align*}
where $\theta \geq 0$ is an arbitrary constant;

${\bf{(A2.2)}}$ $E[\int_0^T |f(s, 0, 0)|^2ds]< + \infty.$

\begin{remark}
In fact, $(A2.1) \Leftrightarrow (A2.1^\prime)$, see Remark 2.2.1 of
Yang (2007).
\end{remark}

By using the fixed point theorem, Yang \cite{Y} (see also Yang and
Elliott \cite{YE}) proved the existence and uniqueness theorem for
GABSDEs:

\begin{theorem}\label{thm:gabsde}
Assume that $f$ satisfies $(A2.1)$ and $(A2.2)$, then for arbitrary
given terminal conditions $(\xi_\cdot, \eta_\cdot) \in
L_\mathcal{F}^2(T, T+K; \mathbb{R}^m)\times L_\mathcal{F}^2(T, T+K;
\mathbb{R}^{m\times d}),$ the GABSDE (\ref{chap5.equation:Yang}) has
a unique solution, i.e., there exists a unique pair of
$\mathcal{F}_t$-adapted processes $(Y_\cdot, Z_\cdot)\in
L_\mathcal{F}^2(0, T+K; \mathbb{R}^m)\times L_\mathcal{F}^2(0, T+K;
\mathbb{R}^{m\times d})$ satisfying (\ref{chap5.equation:Yang}).
\end{theorem}

\begin{remark}\label{remark:gabsde}
It should be mentioned here that, in fact condition $(A2.1)$ can be
weaken to $(A2.1^{\prime\prime})$, which says

${\bf{(A2.1^{\prime\prime})}}$ There exists a constant
$L^{\prime\prime}
> 0$ such that for each $y_\cdot, y_\cdot^\prime \in
L_\mathcal{F}^2(0, T+K; \mathbb{R}^m),$ $z_\cdot, z_\cdot^\prime \in
L_\mathcal{F}^2(0, T+K; \mathbb{R}^{m\times d}),$ the following
holds:
\begin{align*}
& E[\int_0^T |f(s, \{y_r\}_{r\in [s, T+K]}, \{z_r\}_{r\in [s,
T+K]})-f(s, \{y_r^\prime\}_{r\in [s, T+K]}, \{z_r^\prime\}_{r\in [s,
T+K]})|^2ds] \\
& \leq L^{\prime\prime}
E[\int_0^{T+K}(|y_s-y_s^\prime|^2+|z_s-z_s^\prime|^2])ds].
\end{align*}

This can be easily checked from the detailed proofs of the theorem.
\end{remark}

\begin{remark}
Let us give some examples of generator functions satisfying
$(A2.1)$. Assume that $g(\omega, t, y, z): \Omega\times [0, T]\times
\mathbb{R}^m\times \mathbb{R}^{m\times d}\rightarrow \mathbb{R}^m$
is $\mathcal{F}_t$-adapted and Lipschitz in $(y, z)$, i.e., there
exists a constant $L_g>0$ such that $|g(t, y, z)-g(t, y^\prime,
z^\prime)|\leq L_g(|y-y^\prime|+|z-z^\prime|)$ for any $(y, z),
(y^\prime, z^\prime)\in \mathbb{R}^m\times \mathbb{R}^{m\times d}.$
Then we can easily check that $f_1$ and $f_2$ defined below will
satisfy $(A2.1)$:
\begin{align*}
& f_1(t, \{y_r\}_{r\in [t, T+K]}, \{z_r\}_{r\in [t, T+K]}):=g(t,
E^{\mathcal{F}_t}[\int_t^{T+K}y_rdr],
E^{\mathcal{F}_t}[\int_t^{T+K}z_rdr]),\\
& f_2(t, \{y_r\}_{r\in [t, T+K]}, \{z_r\}_{r\in [t,
T+K]}):=E^{\mathcal{F}_t}[g(t, \int_t^{T+K}y_rdr,
\int_t^{T+K}z_rdr)].
\end{align*}
\end{remark}

\subsection{Stochastic Functional Differential Equations}

For each $t\in [0, T]$, let
\begin{align*}
& b(t, \{x_r\}_{r\in [-M, t]}): \Omega \times [0, T]\times
L_\mathcal{F}^2(-M, t; \mathbb{R}^n) \rightarrow L^2
(\mathcal{F}_t; \mathbb{R}^n),\\
& \sigma(t, \{x_r\}_{r\in [-M, t]}): \Omega \times [0, T]\times
L_\mathcal{F}^2(-M, t; \mathbb{R}^n) \rightarrow L^2 (\mathcal{F}_t;
\mathbb{R}^{n\times d}).
\end{align*}

Consider the following stochastic functional differential equation
(SFDE):
\begin{equation}\label{equation:sfde}
\left\{
\begin{tabular}{rlll}
$dX_t$ &=& $b(t, \{X_r\}_{r\in [-M, t]})dt + \sigma(t, \{X_r\}_{r\in [-M, t]})dB_t,$ & $t\in [0, T];$ \vspace{2mm}\\
$X_t$ &=& $\rho_t,$ & $t\in [-M, 0],$
\end{tabular}\right.
\end{equation}
where $\rho_\cdot \in C(-M, 0; \mathbb{R}^n)$.

\begin{definition}
A process $X_\cdot: \Omega\times [-M, T] \rightarrow \mathbb{R}^n$
is called an adapted solution of SFDE (\ref{equation:sfde}) if
$X_\cdot \in L_{\mathcal{F}}^2(-M, T; \mathbb{R}^n)$ and it
satisfies (\ref{equation:sfde}).
\end{definition}

It should be mentioned that, Mohammed \cite{M} has considered
several types of SFDEs, and got the existence and uniqueness result
by using Picard iterations. Here in order to make the paper
self-contained, we will provide a proof by applying the fixed point
theorem rather than Picard iterations.

We impose the following assumption:

${\bf{(A2.3)}}$ There exists a constant $L > 0$ such that for each
$x_\cdot, x_\cdot^\prime \in L_\mathcal{F}^2(-M, T; \mathbb{R}^n),$
the following hold:
\begin{align*}
& E[\int_0^T e^{-\theta s} |b(s, \{x_r\}_{r\in [-M, s]})-b(s,
\{x_r^\prime\}_{r\in[-M, s]})|^2ds] \leq L
E\int_{-M}^{T} e^{-\theta s} |x_s-x_s^\prime|^2ds,\\
& E[\int_0^T e^{-\theta s} |\sigma(s, \{x_r\}_{r\in [-M,
s]})-\sigma(s, \{x_r^\prime\}_{r\in[-M, s]})|^2ds] \leq L
E\int_{-M}^{T} e^{-\theta s} |x_s-x_s^\prime|^2ds,
\end{align*}
where $\theta \geq 0$ is an arbitrary constant;

${\bf{(A2.4)}}$ $b(t, 0)\in L_\mathcal{F}^2(0, T; \mathbb{R}^n)$ and
$\sigma(t, 0)\in L_\mathcal{F}^2(0, T; \mathbb{R}^n)$.

%
%

\begin{remark}\label{remark:sfde}
Let us give some examples of coefficients satisfying $(A2.3)$.
Assume that $p(\omega, t, x): \Omega\times [0, T]\times \mathbb{R}^n
\rightarrow \mathbb{R}^n$ and $q(\omega, t, x): \Omega\times [0,
T]\times \mathbb{R}^n \rightarrow \mathbb{R}^{n\times d}$ are
$\mathcal{F}_t$-adapted and Lipschitz w.r.t. $x$, i.e., there exist
constants $L_p>0$, $L_q>0$ such that $|p(t, x)-p(t, x^\prime)|\leq
L_p |x-x^\prime|$, $|q(t, x)-q(t, x^\prime)|\leq L_q |x-x^\prime|$
for any $x, x^\prime \in \mathbb{R}^n$. Then we can easily check
that $b_1$, $b_2$, $\sigma_1$ and $\sigma_2$ defined below will
satisfy $(A2.3)$:
\begin{align*}
& b_1(t, \{x_r\}_{r\in [-M, t]}):=p(t, \int_{-M}^tx_rdr),\quad
b_2(t, \{x_r\}_{r\in [-M, t]}):=\int_{-M}^t p(r, x_r)dr,\\
& \sigma_1(t, \{x_r\}_{r\in [-M, t]}):=q(t, \int_{-M}^tx_rdr),\quad
\sigma_2(t, \{x_r\}_{r\in [-M, t]}):=\int_{-M}^t q(r, x_r)dr.
\end{align*}
\end{remark}

We now give the existence and uniqueness result for SFDE
(\ref{equation:sfde}).

\begin{theorem}\label{thm:sfde}
Assume that $(A2.3)$ and $(A2.4)$ hold. Then SFDE
(\ref{equation:sfde}) has a unique adapted solution.
\end{theorem}

\begin{proof}
Let $\theta$ be a nonnegative constant. Now we use the following
norm in $L_{\mathcal{F}}^2(-M, T; \mathbb{R}^n)$:
$$\|v(\cdot)\|_{-\theta}:=(E \int_{-M}^T e^{-\theta s}|v(s)|^2ds)^{\frac{1}{2}},$$
which is equivalent to the original norm of $L_{\mathcal{F}}^2(0, T;
\mathbb{R}^n)$. Henceforth we will find that this new norm is more
convenient for us to construct a contraction mapping.

Let $X_\cdot$ be the unique solution of
\begin{equation*}
\left\{
\begin{tabular}{rlll}
$dX_t$ &=& $b(t, \{x_r\}_{r\in [-M, t]})dt + \sigma(t, \{x_r\}_{r\in [-M, t]})dB_t,$ & $t\in [0, T];$ \vspace{2mm}\\
$X_t$ &=& $\rho_t,$ & $t\in [-M, 0],$
\end{tabular}\right.
\end{equation*}
where $x_\cdot \in L_{\mathcal{F}}^2(-M, T; \mathbb{R}^n)$. Now
introduce a mapping $I$ from $L_{\mathcal{F}}^2(0, T; \mathbb{R}^n)$
into itself by $X_\cdot = I(x_\cdot)$.

Let $x_\cdot^\prime$ be another element of $L_{\mathcal{F}}^2(0, T;
\mathbb{R}^n)$, and define $X_\cdot^\prime =I(x_\cdot^\prime)$. We
make the following notations:
\begin{align*}
& \widehat{x}_\cdot=x_\cdot-x_\cdot^\prime,\
\ \ \widehat{X}_\cdot=X_\cdot-X_\cdot^\prime,\\
& \widehat{b}_t = b(t, \{x_{r}\}_{r\in [-M, t]})-b(t,
\{x_{r}^\prime\}_{r\in [-M, t]}),\\
& \widehat{\sigma}_t = \sigma(t, \{x_{r}\}_{r\in [-M, t]})-\sigma(t,
\{x_{r}^\prime\}_{r\in [-M, t]}).
\end{align*}
Then for any $\theta \geq 0$, applying It\^{o}'s formula to
$e^{-\theta t}|\widehat{X}_t|^2$, and taking expectation, we have
$$E e^{-\theta
t}|\widehat{X}_t|^2 = E \int_0^t (-\theta) e^{-\theta s}
|\widehat{X}_s|^2ds + E \int_0^t e^{-\theta s}
|\widehat{\sigma}_s|^2ds + 2 E \int_0^t e^{-\theta s} \widehat{X}_s
\widehat{b}_s ds.$$ This, together with $(A2.3)$, yields
$$E \int_0^T \theta e^{-\theta s} |\widehat{X}_s|^2ds \leq E
\int_{-M}^T e^{-\theta s} (L^2 + \frac{2 L^2}{\theta})
|\widehat{x}_s|^2 ds + E\int_0^T e^{-\theta s} \frac{\theta}{2}
|\widehat{X}_s|^2ds.$$ Thus if we choose $\theta=2L^2 + 2L
\sqrt{L^2+2}$, and note that $\widehat{X}_s \equiv 0$ for $s\in[-M,
0]$, then we deduce
$$E\int_{-M}^T e^{-\theta
s}|\widehat{X}_s|^2ds \leq \frac{1}{2}E\int_{-M}^{T}e^{-\theta
s}|\widehat{x}_s|^2ds,
$$
so that $I$ is a strict contraction on $L_{\mathcal{F}}^2(-M, T;
\mathbb{R}^n)$. It follows by the fixed point theorem that SFDE
(\ref{equation:sfde}) has a unique solution $X_\cdot\in
L_{\mathcal{F}}^2(-M, T; \mathbb{R}^n)$.
\end{proof}

At the end of this part, for the following SFDE, with the same form
as in Chapter II of Mohammed \cite{M}:
\begin{equation}\label{equation:sfde mohammed}
\left\{
\begin{tabular}{rlll}
$dX_t^\prime$ &=& $b^\prime(t, \{X_r^\prime\}_{r\in [t-M, t]})dt + \sigma^\prime(t, \{X_r^\prime\}_{r\in [t-M, t]})dB_t,$ & $t\in [0, T];$ \vspace{2mm}\\
$X_t^\prime$ &=& $\rho_t,$ & $t\in [-M, 0],$
\end{tabular}\right.
\end{equation}
we also give an existence and uniqueness theorem. Since the method
to prove it is similar to Theorem \ref{thm:sfde}, we omit here. For
\begin{align*}
& b^\prime(t, \{x_r\}_{r\in [t-M, t]}): \Omega \times [0, T]\times
L_\mathcal{F}^2(t-M, t; \mathbb{R}^n) \rightarrow L^2
(\mathcal{F}_t; \mathbb{R}^n),\\
& \sigma^\prime(t, \{x_r\}_{r\in [t-M, t]}): \Omega \times [0,
T]\times L_\mathcal{F}^2(t-M, t; \mathbb{R}^n) \rightarrow L^2
(\mathcal{F}_t; \mathbb{R}^{n\times d}),
\end{align*}
we assume that

${\bf{(A2.3^\prime)}}$ There exists a constant $L^\prime > 0$ such
that for each $x_\cdot, x_\cdot^\prime \in L_\mathcal{F}^2(-M, T;
\mathbb{R}^n),$ the following hold:
\begin{align*}
& E[\int_0^T e^{-\theta s} |b^\prime(s, \{x_r\}_{r\in [s-M,
s]})-b^\prime(s, \{x_r^\prime\}_{r\in[s-M, s]})|^2ds] \leq L
E\int_{-M}^{T} e^{-\theta s} |x_s-x_s^\prime|^2ds,\\
& E[\int_0^T e^{-\theta s} |\sigma^\prime(s, \{x_r\}_{r\in [s-M,
s]})-\sigma^\prime(s, \{x_r^\prime\}_{r\in[s-M, s]})|^2ds] \leq L
E\int_{-M}^{T} e^{-\theta s} |x_s-x_s^\prime|^2ds,
\end{align*}
where $\theta \geq 0$ is an arbitrary constant;

${\bf{(A2.4^\prime)}}$ $b^\prime(t, 0)\in L_\mathcal{F}^2(0, T;
\mathbb{R}^n)$ and $\sigma^\prime(t, 0)\in L_\mathcal{F}^2(0, T;
\mathbb{R}^n)$.

\begin{theorem}
Assume that $(A2.3^\prime)$ and $(A2.4^\prime)$ hold. Then SFDE
(\ref{equation:sfde mohammed}) has a unique adapted solution.
\end{theorem}

\section{Fully Coupled Forward-Backward
Stochastic Functional Differential Equations}

In this section, we consider the following fully coupled
forward-backward stochastic functional differential equation
(FBSFDE):
\begin{equation}\label{equation:main}
\left\{
\begin{tabular}{rlll}
$dX_t$ &=& $b(t, \{X_r\}_{r\in [-M, t]}, Y_t, Z_t)dt + \sigma(t, \{X_r\}_{r\in [-M, t]}, Y_t, Z_t)dB_t,$ & $t\in [0, T];$ \vspace{2mm}\\
$-dY_t$ &=& $f(t, X_t, \{Y_r\}_{r\in [t, T+K]}, \{Z_r\}_{r\in [t,
T+K]})dt - Z_t dB_t, $ & $ t\in[0, T];$ \vspace{2mm}
\\
$X_t$ &=& $\rho_t,$ & $t\in [-M, 0];$ \vspace{2mm}
\\
$Y_T$ &=& $\Phi(X_T),$\ \ \ $Y _t\ \ =\ \ \xi_t,$ & $t\in(T, T+K];$
\vspace{2mm}
\\
$Z_t$ &=& $\eta_t, $ & $t\in[T, T+K],$
\end{tabular}\right.
\end{equation}
where
\begin{align*}
& b(t, \cdot, \cdot, \cdot): \Omega \times [0, T]\times
L_\mathcal{F}^2(-M, t; \mathbb{R}^n)\times \mathbb{R}^m \times
\mathbb{R}^{m\times d} \rightarrow L^2
(\mathcal{F}_t; \mathbb{R}^n),\\
& \sigma(t, \cdot, \cdot, \cdot): \Omega \times [0, T]\times
L_\mathcal{F}^2(-M, t; \mathbb{R}^n)\times \mathbb{R}^m \times
\mathbb{R}^{m\times d} \rightarrow L^2
(\mathcal{F}_t; \mathbb{R}^{n\times d}),\\
& f(t, \cdot, \cdot, \cdot): \Omega \times [0, T]\times \mathbb{R}^n
\times L_\mathcal{F}^2(t, T+K; \mathbb{R}^m)\times
L_\mathcal{F}^2(t, T+K; \mathbb{R}^{m\times d})\rightarrow L^2
(\mathcal{F}_t,
\mathbb{R}^m),\\
& \Phi: \Omega\times \mathbb{R}^n \rightarrow \mathbb{R}^m, \
\rho_\cdot \in C(-M, 0; \mathbb{R}^n),\ \xi_\cdot \in
L_{\mathcal{F}}^2(T, T+K; \mathbb{R}^m),\ \eta_\cdot \in
L_{\mathcal{F}}^2(T, T+K; \mathbb{R}^{m\times d}).
\end{align*}

\begin{definition}
A triple of processes $(X_\cdot, Y_\cdot, Z_\cdot): \Omega\times
[-M, T] \times [0, T+K] \times [0, T+K] \rightarrow \mathbb{R}^n
\times \mathbb{R}^m \times \mathbb{R}^{m\times d}$ is called an
adapted solution of FBSFDE (\ref{equation:main}) if $(X_\cdot,
Y_\cdot, Z_\cdot) \in L_{\mathcal{F}}^2(-M, T; \mathbb{R}^n)\times
L_{\mathcal{F}}^2(0, T+K; \mathbb{R}^m)\times L_{\mathcal{F}}^2(0,
T+K; \mathbb{R}^{m\times d})$ and it satisfies FBSFDE
(\ref{equation:main}).
\end{definition}

Given an $m\times n$ full-rank matrix $G$, we use the following
notations:
$$u=\left(
  \begin{array}{c}
    x\\
    y\\
    z
  \end{array}
\right), \left(
\begin{array}{c}
\alpha \\
\beta \\
\gamma
\end{array}
\right)
 =\left(
  \begin{array}{c}
    \{x_r\}_{r\in [-M,\ \cdot]}\\
    \{y_r\}_{r\in [\cdot,\ T+K]}\\
    \{z_r\}_{r\in [\cdot,\ T+K]}
  \end{array}
\right), A(t, u, \alpha, \beta, \gamma)
 =\left(
  \begin{array}{c}
    -G^T f(t, x, \beta, \gamma)\\
    G b(t, \alpha, y, z)\\
    G \sigma(t, \alpha, y, z)
  \end{array}
\right),$$ where $G^T$ denotes the transpose of $G$ and
$G\sigma=(G\sigma_1, G\sigma_2, \cdots, G\sigma_d)$.

Now we introduce the following assumptions:

${\bf{(H3.1)}}$ $E \int_0^T |A(s, u, \alpha, \beta, \gamma)|^2 ds <
+\infty$ for each $(u,\ \alpha,\ \beta,\ \gamma)$;

${\bf{(H3.2)}}$ There exists a constant $L > 0$ such that for each
$x_\cdot, x_\cdot^\prime \in L_\mathcal{F}^2(-M, T; \mathbb{R}^n),$
$y_\cdot, y_\cdot^\prime \in L_\mathcal{F}^2(0, T+K; \mathbb{R}^m),$
$z_\cdot, z_\cdot^\prime \in L_\mathcal{F}^2(0, T+K;
\mathbb{R}^{m\times d}),$ the following hold:
\begin{align*}
& E\int_0^T e^{-\theta s} |b(s, \{x_r\}_{r\in [-M, s]}, y_s,
z_s)-b(s,
\{x_r^\prime\}_{r\in[-M, s]}, y_s^\prime, z_s^\prime)|^2ds \\
& \hspace{1cm} + E\int_0^T e^{-\theta s} |\sigma(s, \{x_r\}_{r\in
[-M, s]}, y_s, z_s)-\sigma(s, \{x_r^\prime\}_{r\in[-M, s]}, y_s^\prime, z_s^\prime)|^2ds \\
& \hspace{1cm} \leq L E\int_{-M}^{T} e^{-\theta s}
|x_s-x_s^\prime|^2ds + L E
\int_0^{T}e^{-\theta s}(|y_s-y_s^\prime|^2+|z_s-z_s^\prime|^2)ds,\\
& E \int_0^T e^{\theta s}|f(s, x_s, \{y_r\}_{r\in [s, T+K]},
\{z_r\}_{r\in [s, T+K]})-f(s, x_s^\prime, \{y_r^\prime\}_{r\in [s,
T+K]},
\{z_r^\prime\}_{r\in [s, T+K]})|^2ds \\
& \hspace{1cm} \leq L E\int_{0}^{T} e^{\theta s}
|x_s-x_s^\prime|^2ds + L E \int_0^{T+K}e^{\theta
s}(|y_s-y_s^\prime|^2+|z_s-z_s^\prime|^2)ds,
\end{align*}
where $\theta \geq 0$ is an arbitrary constant;


${\bf{(H3.3)}}$ $\Phi(x)\in L^2(\mathcal{F}_T; \mathbb{R}^m)$ and it
is uniformly Lipschitz w.r.t. $x\in \mathbb{R}^n$;

${\bf{(H3.4)}}$ $A(\cdot, \cdot, \cdot, \cdot, \cdot)$ and
$\Phi(\cdot)$ satisfy
\begin{align*}
& E \int_0^T \langle A(s, u_s, \alpha_s, \beta_s, \gamma_s)-A(s,
u_s^\prime, \alpha_s^\prime, \beta_s^\prime, \gamma_s^\prime),
u_s-u_s^\prime \rangle ds \\
& \qquad \leq -\lambda_1 E \int_{-M}^T |G \widehat{x}_s|^2ds
-\lambda_2 E\int_0^{T+K} (|G^T
\widehat{y}_s|^2 + |G^T \widehat{z}_s|^2)ds,\\
& \langle \Phi(x)-\Phi(x^\prime), G(x-x^\prime) \rangle \geq \mu |G
\widehat{x}|^2,
\end{align*}
for all $(u,\ \alpha,\ \beta,\ \gamma)$,  $(u^\prime,\
\alpha^\prime,\ \beta^\prime,\ \gamma^\prime)$, $x$ and $x^\prime$,
$\widehat{x}=x-x^\prime$, $\widehat{y}=y-y^\prime$,
$\widehat{z}=z-z^\prime$, where $\lambda_1$, $\lambda_2$ and $\mu$
are given nonnegative constants with $\lambda_1+\lambda_2
>0$, $\lambda_2+\mu > 0$. Moreover, we have $\lambda_1 > 0$, $\mu >
0$ (resp. $\lambda_2 > 0$) when $m > n$ (resp. $n > m$).

We first give the uniqueness theorem.

\begin{theorem}\label{thm:uniqueness}
Assume that $(H3.1)$-$(H3.4)$ hold. Then FBSFDE
(\ref{equation:main}) has at most one adapted solution.
\end{theorem}

\begin{proof}
Suppose that $U_\cdot=(X_\cdot, Y_\cdot, Z_\cdot)$ and
$U_\cdot^\prime=(X_\cdot^\prime, Y_\cdot^\prime, Z_\cdot^\prime)$
are two solutions of FBSFDE (\ref{equation:main}). We denote
$\widehat{U}_\cdot=(\widehat{X}_\cdot, \widehat{Y}_\cdot,
\widehat{Z})_\cdot=(X_\cdot-X_\cdot^\prime, Y_\cdot-Y_\cdot^\prime,
Z_\cdot-Z_\cdot^\prime)$. Applying It\^{o}'s formula to $\langle G
\widehat{X}_t, \widehat{Y}_t \rangle$ and noting $(H3.4)$, we have
\begin{align*}
& E \langle \Phi(X_T)-\Phi(X_T^\prime), G \widehat{X}_T \rangle \\
& = E \int_0^T \langle A(s, U_s, \alpha_s, \beta_s, \gamma_s)-A(s,
U_s^\prime, \alpha_s^\prime, \beta_s^\prime, \gamma_s^\prime),
\widehat{U}_s \rangle ds \\
& \leq -\lambda_1 E \int_{-M}^T |G \widehat{X}_s|^2ds -\lambda_2
E\int_0^{T+K} (|G^T
\widehat{Y}_s|^2 + |G^T \widehat{Z}_s|^2)ds\\
& = -\lambda_1 E \int_0^T |G \widehat{X}_s|^2ds -\lambda_2 E\int_0^T
(|G^T \widehat{Y}_s|^2 + |G^T \widehat{Z}_s|^2)ds,
\end{align*}
where the last equality is due to $X_s=X_s^\prime= \rho_s$ for $s\in
[-M, 0]$ and $(Y_s, Z_s)=(Y_s^\prime, Z_s^\prime)=(\xi_s, \eta_s)$
for $s\in (T, T+K]$.

Together with $(H3.4)$ again, we obtain
$$\lambda_1 E \int_0^T |G \widehat{X}_s|^2ds
+\lambda_2 E\int_0^T (|G^T \widehat{Y}_s|^2 + |G^T
\widehat{Z}_s|^2)ds + \mu |G \widehat{X}_T|^2 \leq 0.$$

For the case when $m > n$, we note that $\lambda_1
>0$ and $\mu >0$. Then it is easy to get that for $s\in [0, T]$, $|G \widehat{X}_s|^2 \equiv 0$, which
implies $\widehat{X}_s \equiv 0$. Thus $X_s \equiv X_s^\prime$, for
$s\in [0, T]$. Then according to Theorem \ref{thm:gabsde} together
with Remark \ref{remark:gabsde}, we know $(Y_s, Z_s)=(Y_s^\prime,
Z_s^\prime)$ for $s\in [0, T]$.

For the case when $n >m$, we note that $\lambda_2 >0$. Then for
$s\in [0, T]$, $|G^T \widehat{Y}_s |^2 \equiv 0$ and $|G^T
\widehat{Z}_s |^2 \equiv 0$, which implies $(Y_s, Z_s) \equiv
(Y_s^\prime, Z_s^\prime)$. Finally, from the uniqueness of SFDEs
(see Theorem \ref{thm:sfde}), it follows that $X_s \equiv
X_s^\prime$ for $s\in [0, T]$.

Similarly to the above two cases, for the case when $m=n$, the
result can be easily obtained.
\end{proof}

From now on, we will mainly study the existence of the solution to
FBSFDE (\ref{equation:main}). For this, we first consider the
following family of FBSFDEs parameterized by $\varepsilon \in [0,
1]$:

\begin{equation}\label{equation:epsilon}
\left\{
\begin{tabular}{ll}
$dX_t^\varepsilon=[(1-\varepsilon)\lambda_2(-G^T Y_t^\varepsilon)+\varepsilon b(t, \{X_r^\varepsilon\}_{r\in [-M, t]}, Y_t^\varepsilon, Z_t^\varepsilon)+\varphi_t]dt$ \vspace{2mm}\\
$\qquad\ \ \ + [(1-\varepsilon)\lambda_2(-G^T Z_t^\varepsilon)+\varepsilon \sigma(t, \{X_r^\varepsilon\}_{r\in [-M, t]}, Y_t^\varepsilon, Z_t^\varepsilon)+\phi_t]dB_t,$\ \ $t\in [0, T];$ \vspace{2mm}\\
$-dY_t^\varepsilon=[(1-\varepsilon)\lambda_1
GX_t^\varepsilon+\varepsilon f(t, X_t^\varepsilon,
\{Y_r^\varepsilon\}_{r\in [t, T+K]}, \{Z_r^\varepsilon\}_{r\in [t,
T+K]})+\psi_t]dt$ \vspace{2mm} \\
$\qquad\qquad - Z_t^\varepsilon dB_t, $\hspace{8.5cm}\ \ \ \
$t\in[0, T];$ \vspace{2mm}
\\
$X_t^\varepsilon=\rho_t,$\hspace{10.5cm}\ \ $t\in [-M, 0];$
\vspace{2mm}
\\
$Y_T^\varepsilon=\varepsilon \Phi(X_T^\varepsilon)+(1-\varepsilon)G
X_T^\varepsilon + \zeta,$\ \ \ $Y _t^\varepsilon\ =\ \xi_t,$
\hspace{4cm} $t\in(T, T+K];$ \vspace{2mm}
\\
$Z_t^\varepsilon=\eta_t, $\hspace{10.5cm}\ \ \ $t\in[T, T+K],$
\end{tabular}\right.
\end{equation}
where $\zeta\in L^2({\mathcal{F}_T}; \mathbb{R}^m)$, $\varphi_\cdot
\in L_{\mathcal{F}}^2(0, T; \mathbb{R}^n)$, $\phi_\cdot \in
L_{\mathcal{F}}^2(0, T; \mathbb{R}^{n\times d})$ and $\psi_\cdot \in
L_{\mathcal{F}}^2(0, T; \mathbb{R}^m)$. It is obvious that the
existence of (\ref{equation:main}) just follows from that of
(\ref{equation:epsilon}) when $\varepsilon=1$.

\begin{lemma}\label{lem:epsilon}
Assume that $(H3.1)$-$(H3.4)$ hold. If for an $\varepsilon_0\in [0,
1)$, there exists a solution $(X_\cdot^{\varepsilon_0},
Y_\cdot^{\varepsilon_0}, Z_\cdot^{\varepsilon_0})$ of FBSFDE
(\ref{equation:epsilon}), then there exists a positive constant
$\delta_0$, such that for each $\delta\in [0, \delta_0]$ there
exists a solution $(X_\cdot^{\varepsilon_0+\delta},
Y_\cdot^{\varepsilon_0+\delta}, Z_\cdot^{\varepsilon_0+\delta})$ of
FBSFDE (\ref{equation:epsilon}) for
$\varepsilon=\varepsilon_0+\delta$.
\end{lemma}

\begin{proof}
Let $u_\cdot =(x_\cdot, y_\cdot, z_\cdot) \in L_{\mathcal{F}}^2(-M,
T; \mathbb{R}^n) \times L_{\mathcal{F}}^2(0, T+K; \mathbb{R}^m)
\times L_{\mathcal{F}}^2(0, T+K; \mathbb{R}^{m\times d})$. Then it
follows that there exists a unique triple $U_\cdot =(X_\cdot,
Y_\cdot, Z_\cdot) \in L_{\mathcal{F}}^2(-M, T; \mathbb{R}^n) \times
L_{\mathcal{F}}^2(0, T+K; \mathbb{R}^m) \times L_{\mathcal{F}}^2(0,
T+K; \mathbb{R}^{m\times d})$ satisfying the following FBSFDE:
\begin{equation*}
\left\{
\begin{tabular}{ll}
$dX_t=[(1-\varepsilon_0)\lambda_2(-G^T Y_t)+\varepsilon_0 b(t, \{X_r\}_{r\in [-M, t]}, Y_t, Z_t)+\varphi_t]dt$ \vspace{2mm}\\
$\qquad \ \ \ + \delta (\lambda_2 G^T y_t + b(t, \{x_r\}_{r\in [-M,
t]}, y_t, x_t)) dt $ \vspace{2mm}\\
$\qquad \ \ \ + [(1-\varepsilon_0)\lambda_2(-G^T Z_t)+\varepsilon_0
\sigma(t, \{X_r\}_{r\in [-M, t]}, Y_t, Z_t)+\phi_t]dB_t$
\vspace{2mm}\\
$\qquad\ \ \ + \delta \lambda_2 (G^T z_t + \sigma(t, \{x_r\}_{r\in [-M, t]}, y_t, z_t))dB_t,$ \hspace{3.5cm} $t\in [0, T];$ \vspace{2mm}\\
$-dY_t=[(1-\varepsilon_0)\lambda_1 GX_t+\varepsilon_0 f(t, X_t,
\{Y_r\}_{r\in [t, T+K]}, \{Z_r\}_{r\in [t,
T+K]})+\psi_t]dt$ \vspace{2mm} \\
$\qquad\qquad +\delta(-\lambda_1 Gx_t + f(t, x_t, \{y_r\}_{r\in [t,
T+K]}, \{z_r\}_{r\in [t,
T+K]}))dt$ \vspace{2mm}\\
$\qquad\qquad - Z_t dB_t, $\hspace{9.3cm} $t\in[0, T];$ \vspace{2mm}
\\
$X_t=\rho_t,$\hspace{11cm} $t\in [-M, 0];$ \vspace{2mm}
\\
$Y_T=\varepsilon_0 \Phi(X_T)+(1-\varepsilon_0)G X_T +\delta
(\Phi(x_T)-G x_T)+ \zeta,$\ \ \ $Y _t\ =\ \xi_t,$ \hspace{6mm}
$t\in(T, T+K];$ \vspace{2mm}
\\
$Z_t=\eta_t, $\hspace{11.2cm}$ t\in[T, T+K].$
\end{tabular}\right.
\end{equation*}

Our objective is to prove that for sufficiently small $\delta$, the
mapping $I_{\varepsilon_0+\delta}$, defined by
$U_\cdot=I_{\varepsilon_0+\delta}(u_\cdot)$ from
$L_{\mathcal{F}}^2(-M, T; \mathbb{R}^n) \times L_{\mathcal{F}}^2(0,
T+K; \mathbb{R}^m) \times L_{\mathcal{F}}^2(0, T+K;
\mathbb{R}^{m\times d})$ into itself, is a contraction mapping.

Let $u_\cdot^\prime = (x_\cdot^\prime, y_\cdot^\prime,
z_\cdot^\prime)$ be another element of $L_{\mathcal{F}}^2(-M, T;
\mathbb{R}^n) \times L_{\mathcal{F}}^2(0, T+K; \mathbb{R}^m) \times
L_{\mathcal{F}}^2(0, T+K; \mathbb{R}^{m\times d})$ and define
$U_\cdot^\prime = I_{\varepsilon_0+\delta} (u_\cdot^\prime)$. We
make the following notations:
\begin{align*}
& \widehat{u}_\cdot=(\widehat{x}_\cdot, \widehat{y}_\cdot,
\widehat{z}_\cdot)=(x_\cdot-x_\cdot^\prime, y_\cdot-y_\cdot^\prime,
z_\cdot-z_\cdot^\prime),\\
& \widehat{U}_\cdot=(\widehat{X}_\cdot, \widehat{Y}_\cdot,
\widehat{Z}_\cdot)=(X_\cdot-X_\cdot^\prime, Y_\cdot-Y_\cdot^\prime,
Z_\cdot-Z_\cdot^\prime),\\
&  \widehat{b}_t = b(t, \{x_{r}\}_{r\in [-M, t]}, y_t, z_t)-b(t,
\{x_{r}^\prime\}_{r\in [-M, t]}, y_t, z_t),\\
& \widehat{\sigma}_t = \sigma(t, \{x_{r}\}_{r\in [-M, t]}, y_t,
z_t)-\sigma(t, \{x_{r}^\prime\}_{r\in [-M, t]}, y_t, z_t),\\
& \widehat{f}_t = f(t, x_t, \{y_r\}_{r\in [t, T+K]}, \{z_r\}_{r\in
[t, T+K]})-f(t, x_t^\prime, \{y_r^\prime\}_{r\in [t, T+K]},
\{z_r^\prime\}_{r\in [t, T+K]}).
\end{align*}
Apply It\^{o}'s formula to $\langle G \widehat{X}_t, \widehat{Y}_t
\rangle$, and take expectation,
\begin{align*}
& \varepsilon_0 E \langle \Phi(X_T)-\Phi(X_T^\prime), G
\widehat{X}_T \rangle + (1-\varepsilon_0) E |G \widehat{X}_T|^2 +
\delta E \langle \Phi(x_T)-\Phi(x_T^\prime)
-G \widehat{x}_T, G\widehat{X}_T \rangle ds \\
& = E \int_0^T \varepsilon_0 \langle A(s, U_s, \alpha_s, \beta_s,
\gamma_s)-A(s, U_s^\prime, \alpha_s^\prime, \beta_s^\prime,
\gamma_s^\prime),
\widehat{U}_s \rangle ds \\
& \quad -(1-\varepsilon_0) E \int_0^T (\lambda_1 \langle G\widehat{X}_s, G\widehat{X}_s\rangle + \lambda_2 \langle G^T\widehat{Y}_s, G^T\widehat{Y}_s\rangle + \lambda_2 \langle G^T\widehat{Z}_s, G^T\widehat{Z}_s\rangle)ds\\
& \quad + \delta E\int_0^T (\lambda_1 \langle G\widehat{X}_s, G\widehat{x}_s \rangle + \lambda_2 \langle G^T\widehat{Y}_s, G^T\widehat{y}_s \rangle + \lambda_2 \langle G^T\widehat{Z}_s, G^T\widehat{z}_s \rangle \\
& \hspace{2.2cm}  + \langle \widehat{X}_s, -G^T \widehat{f}_s\rangle
+ \langle G^T \widehat{Y}_s, \widehat{b}_s\rangle +\langle
\widehat{Z}_s, G \widehat{\sigma}_s\rangle)ds.
\end{align*}
From $(H3.1)$-$(H3.4)$, we have
\begin{equation}\label{estimate:XYZ}
\begin{tabular}{rl}
& $(\varepsilon_0 \mu + (1-\varepsilon_0)) E |G\widehat{X}_T|^2+
\lambda_1 E \int_{-M}^T |G \widehat{X}_s|^2ds + \lambda_2 E
\int_0^{T+ K} (|G^T \widehat{Y}_s|^2+|G^T \widehat{Z}_s|^2)ds$ \vspace{3mm} \\
& $\leq C_1 \delta E \int_{-M}^T (|\widehat{X}_s|^2+
|\widehat{x}_s|^2)ds + C_1 \delta E\int_0^{T+K} (|\widehat{Y}_s|^2+
|\widehat{y}_s|^2 + |\widehat{Z}_s|^2+
|\widehat{z}_s|^2)ds$ \vspace{3mm} \\
& $\quad + C_1 \delta E |\widehat{X}_T|^2 + C_1 \delta E
|\widehat{x}_T|^2.$
\end{tabular}
\end{equation}
Here the constant $C_1$ depends on $G$, $L$, $\lambda_1$,
$\lambda_2$.

Next we will give two other estimates. On the one hand, similarly to
the proof of Theorem \ref{thm:sfde}, for any $\theta \geq 0$, by
applying It\^{o}'s formula to $e^{-\theta t}|\widehat{X}_t|^2$, we
have
\begin{align*}
& E e^{-\theta T} |\widehat{X}_T|^2 = E \int_0^T (-\theta)
e^{-\theta s} |\widehat{X}_s|^2ds + E \int_0^T e^{-\theta s}
|\widehat{\sigma}_s|^2ds + 2 E \int_0^T e^{-\theta s} \widehat{X}_s
\widehat{b}_s ds\\
& \leq E \int_0^T (-\frac{\theta}{2}) e^{-\theta s}
|\widehat{X}_s|^2ds + E \int_0^T e^{-\theta s}
|\widehat{\sigma}_s|^2ds + \frac{2}{\theta} E \int_0^T e^{-\theta s}
|\widehat{b}_s|^2 ds,
\end{align*}
where
\begin{align*}
\widehat{b}_s = & (1-\varepsilon_0) \lambda_2 (-G^T \widehat{Y}_s) +
\varepsilon_0 (b(s, \{X_r\}_{r\in[-M, s]}, Y_s, Z_s)-b(s,
\{X_r^\prime\}_{r\in[-M, s]}, Y_s^\prime, Z_s^\prime)) \\
& + \delta (\lambda_2 G^T \widehat{y}_s + b(s, \{x_r\}_{r\in[-M, s]}, y_s, z_s)-b(s, \{x_r^\prime\}_{r\in[-M, s]}, y_s^\prime, z_s^\prime)),\\
\widehat{\sigma}_s = & (1-\varepsilon_0) \lambda_2 (-G^T
\widehat{Z}_s) + \varepsilon_0 (\sigma(s, \{X_r\}_{r\in[-M, s]},
Y_s, Z_s)-\sigma(s, \{X_r^\prime\}_{r\in[-M, s]}, Y_s^\prime,
Z_s^\prime)) \\
& + \delta (\lambda_2 G^T \widehat{z}_s + \sigma(s,
\{x_r\}_{r\in[-M, s]}, y_s, z_s)-\sigma(s, \{x_r^\prime\}_{r\in[-M,
s]}, y_s^\prime, z_s^\prime)).
\end{align*}
According to $(H3.2)$,
\begin{align*}
& E \int_0^T e^{-\theta s}|\widehat{b}_s|^2ds \\
& \leq
4(1-\varepsilon_0)^2
\lambda_2^2 E \int_0^T e^{-\theta s} |G^T \widehat{Y}_s|^2ds \\
& +4 \varepsilon_0^2 L (E\int_{-M}^T e^{-\theta
s}|\widehat{X}_s|^2ds + E\int_{0}^{T} e^{-\theta s}
|\widehat{Y}_s|^2ds+E\int_{0}^{T}
e^{-\theta s} |\widehat{Z}_s|^2ds) \\
& +4 \delta^2 \lambda_2^2 E \int_0^T e^{-\theta s} |G^T
\widehat{y}_s|^2ds
\\
& \leq 4 \lambda_2^2 E \int_0^T e^{-\theta s} |G^T \widehat{Y}_s|^2ds \\
& +4 L (E\int_{-M}^T e^{-\theta s}|\widehat{X}_s|^2ds +
E\int_{0}^{T} e^{-\theta s} |\widehat{Y}_s|^2ds+E\int_{0}^{T}
e^{-\theta s} |\widehat{Z}_s|^2ds) \\
& +4 \delta^2 \lambda_2^2 E \int_0^T e^{-\theta s} |G^T
\widehat{y}_s|^2ds
\\
& + 4\delta^2 L (E\int_{-M}^T e^{-\theta s} |\widehat{x}_s|^2ds +
E\int_{0}^{T} e^{-\theta s} |\widehat{y}_s|^2ds+E\int_{0}^{T}
e^{-\theta s} |\widehat{z}_s|^2ds),
\end{align*}
and similarly,
\begin{align*}
& E \int_0^T e^{-\theta s} |\widehat{\sigma}_s|^2ds \\
& \leq 4 \lambda_2^2 E \int_0^T e^{-\theta s} |G^T \widehat{Z}_s|^2ds \\
& +4 L (E\int_{-M}^T e^{-\theta s} |\widehat{X}_s|^2ds +
E\int_{0}^{T} e^{-\theta s} |\widehat{Y}_s|^2ds+E\int_{0}^{T}
e^{-\theta s} |\widehat{Z}_s|^2ds) \\
& +4 \delta^2 \lambda_2^2 E \int_0^T e^{-\theta s} |G^T
\widehat{z}_s|^2ds
\\
& + 4\delta^2 L (E\int_{-M}^T e^{-\theta s} |\widehat{x}_s|^2ds +
E\int_{0}^{T} e^{-\theta s} |\widehat{y}_s|^2ds+E\int_{0}^{T}
e^{-\theta s} |\widehat{z}_s|^2ds).
\end{align*}
Thus,
\begin{align*}
& \frac{\theta}{2} E \int_0^T  e^{-\theta s} |\widehat{X}_s|^2ds \\
& \leq E \int_0^T e^{-\theta s} |\widehat{\sigma}_s|^2ds +
\frac{2}{\theta} E \int_0^T e^{-\theta s} |\widehat{b}_s|^2 ds
\\
& \leq 4 L (1+\frac{2}{\theta}) E \int_{-M}^T
e^{-\theta s} |\widehat{X}_s|^2ds \\
& +C_2 E \int_0^{T+K} (e^{-\theta s} |\widehat{Y}_s|^2+ e^{-\theta s} |\widehat{Z}_s|^2)ds \\
& + C_2 \delta^2 E \int_{-M}^T e^{-\theta s} |\widehat{x}_s|^2ds +
C_2 \delta^2 E \int_0^{T+K} e^{-\theta s}
(|\widehat{y}_s|^2+|\widehat{z}_s|^2)ds.
\end{align*}
Choosing $\theta$ sufficiently large, we can easily get the
following estimate:
\begin{equation}\label{estimate:X}
\begin{tabular}{rl}
& $E\int_{-M}^T |\widehat{X}_s|^2 ds$ \vspace{3mm}\\
& $\leq C_2 \delta^2 E \int_{-M}^T |\widehat{x}_s|^2ds + C_2
\delta^2 E \int_0^{T+K} (|\widehat{y}_s|^2+ |\widehat{z}_s|^2)ds +
C_2 E \int_0^{T+K} (|\widehat{Y}_s|^2+ |\widehat{Z}_s|^2)ds.$
\end{tabular}
\end{equation}
Here the constant $C_2$ depends on $G$, $L$, $\lambda_2$.

On the other hand, for $(\widehat{Y}_\cdot, \widehat{Z}_\cdot)$,
thanks to the estimate of BSDEs, together with $(H3.2)$, we can
easily derive
\begin{equation}\label{estimate:YZ}
\begin{tabular}{rl}
& $E\int_0^{T+K} (|\widehat{Y}_s|^2+ |\widehat{Z}_s|^2)ds$ \vspace{3mm} \\
& $\leq C_3 \delta^2 E \int_{-M}^T |\widehat{x}_s|^2ds + C_3
\delta^2 E \int_0^{T+K} (|\widehat{y}_s|^2+ |\widehat{z}_s|^2)ds +
C_3 \delta^2 E
|\widehat{x}_T|^2$ \vspace{3mm}\\
& $+ C_3 E\int_{-M}^T |\widehat{X}_s|^2ds + C_3 E
|\widehat{X}_T|^2.$
\end{tabular}
\end{equation}
Here the constant $C_3$ depends on $G$, $L$, $\lambda_1$.

Now combining the above three estimates
(\ref{estimate:XYZ})-(\ref{estimate:YZ}), and noting the fact that
$\mu >0$ implies $\varepsilon_0\mu + (1-\varepsilon_0) >0$, we can
easily check that, whenever $\lambda_1
>0$, $\mu >0$, $\lambda_2 \geq 0$ or $\lambda_1 \geq 0$, $\mu \geq 0$,
$\lambda_2 > 0$, the following always holds:
\begin{align*}
& E \int_{-M}^T |\widehat{X}_s|^2ds + E\int_0^{T+K}
(|\widehat{Y}_s|^2
+|\widehat{Z}_s|^2)ds + E |\widehat{X}_T|^2 \\
& \leq C (\delta+\delta^2) \left(E \int_{-M}^T |\widehat{x}_s|^2ds +
E\int_0^{T+K} (|\widehat{y}_s|^2 +|\widehat{z}_s|^2)ds + E
|\widehat{x}_T|^2\right),
\end{align*} where the
constant $C$ depends on $C_1$, $C_2$, $C_3$, $\lambda_1$,
$\lambda_2$, $\mu$.

Thus if we choose $\delta_0=\min\{1, \frac{1}{4C}\}$, we can clearly
see that, for each $\delta \in [0, \delta_0]$, the mapping
$I_{\varepsilon_0+\delta}$ is a strict contraction on
$L_{\mathcal{F}}^2(-M, T; \mathbb{R}^n) \times L_{\mathcal{F}}^2(0,
T+K; \mathbb{R}^m)\times L_{\mathcal{F}}^2(0, T+K;
\mathbb{R}^{m\times d})$ in the sense that
\begin{align*}
& E \int_{-M}^T |\widehat{X}_s|^2ds + E\int_0^{T+K} (|\widehat{Y}|^2
+|\widehat{Z}_s|^2)ds + E |\widehat{X}_T|^2 \\
& \leq \frac{1}{2} \left(E \int_{-M}^T |\widehat{x}_s|^2ds +
E\int_0^{T+K} (|\widehat{y}|^2 +|\widehat{z}_s|^2)ds + E
|\widehat{x}_T|^2\right).
\end{align*}

Then it follows by the fixed point theorem that the mapping
$I_{\varepsilon_0 +\delta}$ has a unique fixed point
$U_\cdot^{\varepsilon_0+\delta}=(X_\cdot^{\varepsilon_0+\delta},
Y_\cdot^{\varepsilon_0+\delta}, Z_\cdot^{\varepsilon_0+\delta})$,
which is the unique solution of (\ref{equation:epsilon}) for
$\varepsilon = \varepsilon_0 +\delta$.
\end{proof}

Now we give the main result of this part.

\begin{theorem}\label{thm:main}
Assume that $(H3.1)$-$(H3.4)$ hold. Then there exists a unique
adapted solution $(X_\cdot, Y_\cdot, Z_\cdot)$ of FBSFDE
(\ref{equation:main}).
\end{theorem}
\begin{proof}
The uniqueness is an immediate result from Theorem
\ref{thm:uniqueness}. Next we prove the existence.

Note that FBSFDE (\ref{equation:epsilon}) for $\varepsilon=0$ admits
a unique solution (see Theorem 2.6 in \cite{PW}). Thus from Lemma
\ref{lem:epsilon}, there exists a positive constant $\delta_0$ such
that, for each $\delta\in [0, \delta_0]$, (\ref{equation:epsilon})
for $\varepsilon=\varepsilon_0+\delta$ admits a unique solution.
Repeat this process for $N$ times with $1\leq N\delta_0
<1+\delta_0$, then we can obtain that particularly for
$\varepsilon=1$ with $\phi_\cdot=0$, $\varphi_\cdot=0$, $\psi_0=0$
and $\zeta=0$, (\ref{equation:epsilon}) has a unique solution, i.e.,
FBSFDE (\ref{equation:main}) has a unique solution.
\end{proof}

\section{Quadratic Optimal Control Problem for Functional Stochastic Systems}

Let $\rho_\cdot \in C(-M, 0; \mathbb{R}^n)$, and let $v_\cdot$ be an
admissible control process, i.e., an $\mathcal{F}_t$-adapted square
integrable process taking values in a given subset of
$\mathbb{R}^k$. Then we consider the following control system:
\begin{equation}\label{equation:control forward}
\left\{
\begin{tabular}{rlll}
$dX_t$ &=& $(A_t \int_{-M}^t X_sds + C_t v_t)dt + (D_t \int_{-M}^t X_sds + F_t v_t) dB_t,$ & $t\in [0, T];$ \vspace{2mm}\\
$X_t$ &=& $\rho_t,$ & $t\in [-M, 0],$
\end{tabular}\right.
\end{equation}
where $A_\cdot$, $C_\cdot$, $D_\cdot$ and $F_\cdot$ are bounded
progressively measurable matrix-valued processes with appropriate
dimensions. Then according to Theorem \ref{thm:sfde} and Remark
\ref{remark:sfde}, SFDE (\ref{equation:control forward}) admits a
unique solution.

The classical quadratic optimal control problem is to minimize the
cost function
$$J(v_\cdot)=\frac{1}{2}E [\int_0^T (\langle R_t X_t, X_t \rangle + \langle N_t v_t, v_t \rangle)dt + \langle Q X_T, X_T \rangle],$$
where $Q$ is an $\mathcal{F}_T$-measurable nonnegative symmetric
bounded matrix, $R_\cdot$ is an $n \times n$ nonnegative symmetric
bounded progressively measurable matrix-valued process, $N_\cdot$ is
an $k\times k$ positive symmetric bounded progressively measurable
matrix-valued process and its inverse $N_\cdot^{-1}$ is also
bounded.

The following theorem tells us that, for the above optimal control
problem, we can find the explicit form of the optimal control
$u_\cdot$ satisfying $$J(u_\cdot)=\inf_{v_\cdot} J(v_\cdot),$$ by
means of the fully coupled forward-backward stochastic functional
differential equations.

\begin{theorem}
The process $$u_t=-N_t^{-1} (C_t^T Y_t + F_t^T Z_t),\ \ \ t\in [0,
T]$$ is the unique optimal control which satisfies
$$J(u_\cdot)=\inf_{v_\cdot} J(v_\cdot),$$ where $(X_\cdot, Y_\cdot,
Z_\cdot)$ is the unique solution of the following FBSFDE:
\begin{equation}\label{equation:control forward backward}
\left\{
\begin{tabular}{rlll}
$dX_t$ &=& $[A_t \int_{-M}^t X_sds - C_t N_t^{-1} (C_t^T Y_t +
F_t^T Z_t)]dt$ & \vspace{2mm} \\
& & $+ [D_t \int_{-M}^t X_sds - F_t N_t^{-1} (C_t^T Y_t + F_t^T Z_t)]dB_t,$ & $t\in [0, T];$ \vspace{2mm}\\
$-dY_t$ &=& $[E^{\mathcal{F}_t} (\int_t^{T+K} A_s^T Y_s ds) +
E^{\mathcal{F}_t} (\int_t^{T+K}C_s^T Z_sds) + R_t X_t]dt$ &
\vspace{2mm} \\
& & $- Z_t dB_t, $ & $ t\in[0, T];$ \vspace{2mm}
\\
$X_t$ &=& $\rho_t,$ & $t\in [-M, 0];$ \vspace{2mm}
\\
$Y_T$ &=& $Q X_T,$\ \ \ $Y _t\ \ =\ \ 0,$ & $t\in(T, T+K];$
\vspace{2mm}
\\
$Z_t$ &=& $0, $ & $t\in[T, T+K].$
\end{tabular}\right.
\end{equation}
\end{theorem}
\begin{proof}
From Theorem \ref{thm:main}, we know that FBSFDE
(\ref{equation:control forward backward}) admits a unique solution
$(X_\cdot, Y_\cdot, Z_\cdot)$. Denote the unique solution of SFDE
(\ref{equation:control forward}) by $X_\cdot^v$ for the control
$v_\cdot$.

Applying It\^{o}'s formula to $\langle X_t^v-X_t, Y_t \rangle$, and
taking expectation, we have
\begin{align*}
& E\langle X_T^v-X_T, Y_T \rangle \\
& =- E \int_0^T [\langle E^{\mathcal{F}_t}(\int_t^{T+K} A_s^T
Y_sds)+ E^{\mathcal{F}_t}(\int_t^{T+K} D_s^T Z_sds) + R_t X_t, X_t^v-X_t \rangle]dt \\
& \quad + E \int_0^T (\langle A_t \int_{-M}^{t} (X_s^v-X_s) ds,
Y_t \rangle + \langle C_t (v_t-u_t), Y_t \rangle) dt \\
& \quad + E \int_0^T (\langle D_t \int_{-M}^t (X_s^v - X_s)ds, Z_t
\rangle + \langle F_t(v_t-u_t, Z_t) \rangle) dt.
\end{align*}
Note that
\begin{align*}
& E \int_0^T (\langle A_t \int_{-M}^{t} (X_s^v-X_s) ds, Y_t \rangle
- \langle E^{\mathcal{F}_t}(\int_t^{T+K} A_s^T Y_sds), X_t^v-X_t
\rangle )dt \\
& = E \int_0^T \langle A_t \int_{-M}^{t} (X_s^v-X_s) ds, Y_t
\rangle dt - E \int_0^T \langle \int_t^{T+K} A_s^T Y_sds, X_t^v-X_t \rangle dt \\
& = E \int_0^T \langle A_t \int_{-M}^{t} (X_s^v-X_s) ds, Y_t
\rangle dt \\
& \qquad - E \int_0^T \langle A_t \int_{0}^{t} (X_s^v-X_s) ds, Y_t
\rangle dt - E \int_T^{T+K} \langle A_t \int_{0}^{T} (X_s^v-X_s) ds,
Y_t \rangle dt \\
& = E \int_0^T \langle A_t \int_{-M}^{0} (X_s^v-X_s) ds, Y_t \rangle
dt - E \int_T^{T+K} \langle A_t \int_{0}^{T} (X_s^v-X_s) ds,
Y_t \rangle dt \\
& = 0,
\end{align*}
and similarly,
\begin{align*}
E \int_0^T (\langle D_t \int_{-M}^{t} (X_s^v-X_s) ds, Z_t \rangle -
\langle E^{\mathcal{F}_t}(\int_t^{T+K} D_s^T Z_sds), X_t^v-X_t
\rangle )dt = 0.
\end{align*}
Combining the above three equalities, we have
\begin{align*}
E\langle X_T^v-X_T, Y_T \rangle = E \int_0^T (\langle -R_t X_t,
X_t^v-X_t \rangle + \langle C_t (v_t-u_t), Y_t \rangle + \langle
F_t(v_t-u_t), Z_t \rangle) dt,
\end{align*}
which implies
\begin{equation}\label{eqation:control}
E\langle X_T^v-X_T, Y_T \rangle + E \int_0^T \langle R_t X_t,
X_t^v-X_t \rangle = E \int_0^T ( \langle C_t (v_t-u_t), Y_t \rangle
+ \langle F_t(v_t-u_t), Z_t \rangle) dt.
\end{equation}

On the other hand,
\begin{align*}
& J(v_\cdot)-J(u_\cdot) \\
& = \frac{1}{2}E[\int_0^T(\langle R_t X_t^v, X_t^v \rangle - \langle
R_t X_t, X_t \rangle + \langle N_t v_t, v_t \rangle - \langle N_t
u_t, u_t \rangle)dt \\
& \quad \qquad + \langle Q X_T^v, X_T^v \rangle - \langle Q X_T, X_T
\rangle] \\
& = \frac{1}{2} E [\int_0^T(\langle R_t (X_t^v-X_t), X_t^v-X_t
\rangle
+2 \langle R_t X_t, X_t^v-X_t \rangle \\
& \qquad \qquad + \langle N_t (v_t-u_t, v_t-u_t \rangle +2 \langle
N_t u_t, v_t-u_t \rangle)dt \\
& \quad \qquad + \langle Q (X_T^v-X_T), X_T^v-X_T \rangle +2 \langle
Q X_T, X_T^v-X_T \rangle] \\
& \geq E \int_0^T( \langle R_t X_t, X_t^v-X_t \rangle + \langle N_t
u_t, v_t-u_t \rangle)dt + \langle Q X_T, X_T^v-X_T \rangle],
\end{align*}
where the last inequality is due to the positivity of $N_\cdot$, and
the nonnegativity of $R_\cdot$ and $Q$. Then together with
(\ref{eqation:control}), and noting that $Y_T=Q X_T$, we obtain
\begin{align*}
& J(v_\cdot)-J(u_\cdot) \\
& \geq E \int_0^T( \langle R_t X_t, X_t^v-X_t \rangle + \langle N_t
u_t, v_t-u_t \rangle)dt + \langle Q X_T, X_T^v-X_T \rangle] \\
& = E \int_0^T (\langle C_t (v_t-u_t), Y_t \rangle + \langle
F_t(v_t-u_t), Z_t \rangle + \langle N_t u_t, v_t-u_t \rangle )dt \\
& =0.
\end{align*}
Hence, $u_t=-N_t^{-1} (C_t^T Y_t + F_t^T Z_t)$ is an optimal
control.

Moreover, the optimal control is unique. In fact, assume that
$u_\cdot$ and $u_\cdot^\prime$ are both optimal controls, and denote
$J(u_\cdot)=J(u_\cdot^\prime) \triangleq J \geq 0$. The
corresponding trajectories are $X_\cdot$ and $X_\cdot^\prime$. It is
easy to check that for the control
$\frac{u_\cdot+u_\cdot^\prime}{2}$, the trajectory is
$\frac{X_\cdot+X_\cdot^\prime}{2}$. Then,
\begin{align*}
2J & = J(u_\cdot) + J(u_\cdot^\prime) \\
& = \frac{1}{2}E [\int_0^T (\langle R_t X_t, X_t \rangle + \langle
N_t u_t, u_t \rangle)dt + \langle Q X_T, X_T \rangle] \\
& \quad + \frac{1}{2}E [\int_0^T (\langle R_t X_t^\prime, X_t^\prime
\rangle + \langle N_t u_t^\prime, u_t^\prime \rangle)dt + \langle Q
X_T^\prime, X_T^\prime \rangle] \\
& = 2 J(\frac{u_\cdot+u_\cdot^\prime}{2}) + E [\int_0^T (\langle R_t
\frac{X_t-X_t^\prime}{2}, \frac{X_t-X_t^\prime}{2} \rangle + \langle
N_t \frac{u_t-u_t^\prime}{2}, \frac{u_t-u_t^\prime}{2} \rangle)dt]
\\
& \quad + E \langle Q \frac{X_T-X_T^\prime}{2},
\frac{X_T-X_T^\prime}{2} \rangle \\
& \geq 2J + E \int_0^T \langle N_t \frac{u_t-u_t^\prime}{2},
\frac{u_t-u_t^\prime}{2} \rangle dt,
\end{align*}
where the last inequality is due to the nonnegativity of $R_\cdot$
and $Q$.

Therefore $u_\cdot =u_\cdot^\prime$, thanks to the positivity of
$N_\cdot$.
\end{proof}

\begin{remark}
It should be mentioned here that, the method we applied to prove the
uniqueness above is in fact a classical method, readers are referred
to \cite{CW} or \cite{W}.
\end{remark}



%
%

\end{document}